\documentclass[12pt]{amsart}
\usepackage[utf8]{inputenc}
\usepackage[T1]{fontenc}
\usepackage{amsmath,amssymb,amsthm,graphics,amscd,mathrsfs}
\usepackage{amscd, amsfonts, amssymb, amsthm}
\usepackage{fullpage, verbatim}
\usepackage[colorlinks, breaklinks, linkcolor=blue]{hyperref}
\usepackage{breakurl}
\usepackage{array}
\usepackage{latexsym}
\usepackage[shortlabels]{enumitem}
\usepackage{parskip}

\newcommand{\ft}{\operatorname{T}}

\newcounter{myrow}

\usepackage{xcolor}
\usepackage{subcaption}
\usepackage{tikz-cd}
\usepackage{tikz}
\usetikzlibrary{arrows}
\usetikzlibrary{shapes}
\usepackage{tkz-graph}
\usetikzlibrary{decorations}
\usetikzlibrary{arrows,decorations.markings}
\usepackage{multirow}
\tikzstyle{v} = [circle, draw, inner sep=2pt, minimum size=3pt, fill=black]
\tikzstyle{l} = [rectangle, draw, rounded corners]

\usepackage{caption}
\usetikzlibrary{calc}

\renewcommand{\phi}{\varphi}
\renewcommand{\rho}{\varrho}
\renewcommand{\epsilon}{\varepsilon}
\newcommand{\leer}{\varnothing}
\renewcommand{\tilde}{\widetilde}

\newcommand\CG{{\mathcal G}}

\newcommand\CH{{\mathcal H}}

\newcommand\BBC{{\mathbb C}}

\newcommand\BBG{{\mathbb G}}

\newcommand\BBQ{{\mathbb Q}}

\newcommand\BBZ{{\mathbb Z}}

\newcommand\id{{\mathrm {id}}}

\renewcommand\mod{{/\!\!/}}

\newcommand\G{{\operatorname{\mathsf G}}}
\newcommand\Gm{{{\BBG}_m}}

\newcommand\Ext{{\operatorname{Ext}}}

\newcommand\GL{\operatorname{GL}}

\newcommand\SL{\operatorname{SL}}
\newcommand\Sp{\operatorname{Sp}}
\newcommand\Spin{\operatorname{Spin}}
\newcommand\SOO{\operatorname{SO}}
\renewcommand\O{\operatorname{O}}

\newcommand\rk{\operatorname{rk}}

\usepackage{aliascnt}

\numberwithin{equation}{section}

\newaliascnt{theorem}{equation}
\newaliascnt{lemma}{equation}
\newaliascnt{corollary}{equation}
\newaliascnt{proposition}{equation}
\newaliascnt{definition}{equation}
\newaliascnt{remark}{equation}
\newaliascnt{example}{equation}

\newtheorem{theorem}[theorem]{Theorem}
\newtheorem{lemma}[lemma]{Lemma}

\theoremstyle{definition}
\newtheorem{definition}[definition]{Definition}
\newtheorem{remark}[remark]{Remark}

\aliascntresetthe{theorem}
\aliascntresetthe{lemma}
\aliascntresetthe{corollary}
\aliascntresetthe{proposition}
\aliascntresetthe{definition}
\aliascntresetthe{remark}
\aliascntresetthe{example}

\newcommand\ul[1]{\underline{#1}}

\usepackage{cleveref}


\begin{document}

\title[Spherical subgroups in reductive algebraic groups]
{Spherical subgroups in reductive algebraic groups}

\dedicatory{Dedicated to Professor Jean-Pierre Serre on the occasion of his 100th birthday \\ with our deepest admiration and respect.}

\author[F. Knop]{Friedrich Knop}
\email{friedrich.knop@fau.de}
\address
{Department Mathematik,
FAU Erlangen-Nürnberg,
Cauerstraße 11,
91058 Erlangen, Germany}

\author[G. Röhrle]{Gerhard Röhrle}
\email{gerhard.roehrle@rub.de}
\address
{Fakultät für Mathematik,
Ruhr-Universität Bochum,
D-44780 Bochum, Germany}


\allowdisplaybreaks

\begin{abstract}
	In \cite{KR}, we determined all spherical affine homogeneous varieties for simple algebraic groups in arbitrary characteristic. The present paper extends this classification to semisimple groups. This generalizes work done independently by Brion and Mikityuk in characteristic zero. Our primary approach relies on a lifting lemma to characteristic zero, enabling us to directly apply Brion's and Mikityuk's results.
\end{abstract}

\maketitle


\section{Introduction}

Let $G$ be a reductive algebraic group defined over an algebraically
closed field $k$ of characteristic $p\geqslant0$.  A closed subgroup $H$ of
$G$ is called \emph{spherical} if it has a dense orbit on the
flag variety $G/B$ of $G$.  Alternatively, $B$ acts on $G/H$ with an
open orbit.  Accordingly, a $G$-variety with this property is
also called \emph{spherical}.  If $H$ is spherical in $G$, then we
also speak of $(G,H)$ as a \emph{spherical pair}.
 
The purpose of this paper is to classify all connected reductive
spherical subgroups of reductive groups in arbitrary
characteristic, thereby generalizing the classifications of Brion
\cite{Bri87} and Mikityuk \cite[\S 5]{mik} to positive characteristic.

The classification of such subgroups in the case when $G$ is simple
was obtained by Krämer in \cite{Kr}.  In \cite{KR}, we extended
Krämer's classification to arbitrary positive characteristic.

The class of reductive spherical subgroups is of particular
importance. This is shown by the fact that Krämer's list permeates
much of the theory of spherical varieties in characteristic zero. In
particular, these kinds of subgroups provide many of the building
blocks for arbitrary spherical subgroups (see, e.g., Bravi--Pezzini
\cite{BrPezz}). We expect reductive spherical subgroups to play a
similar role for arbitrary $p$. For instance, the results from
\cite{KR} were already used in \cite{Knop5} to list all spherical
subgroups of rank $1$, which is crucial for the theory of general
spherical varieties.

For $p\ne2$, the class of reductive spherical subgroups includes all
symmetric subgroups, i.e., subgroups that are fixed point sets of an
involutive automorphism of $G$ (see e.g., Springer \cite{springer}).
Nevertheless, for $p=2$, symmetric subgroups are ill-behaved. Thus,
reductive spherical subgroups seem to be the correct replacement.

As a consequence of the Bruhat decomposition, any reductive group $G$
gives rise to a spherical pair -- this is one of the most fundamental examples
of a spherical subgroup (indeed, of a symmetric subgroup when
$p \ne 2$) in the non-simple case.  For example, let $G$ be a
reductive group and $B$ a Borel subgroup of $G$.  By the Bruhat
decomposition, the orbits of the diagonal action of $G$ on $G/B \times G/B$
are in bijective correspondence with the Weyl group of $G$.  In
particular, there is a dense $G$-orbit in $G/B \times G/B$, and
consequently, the diagonal subgroup in $G \times G$ is spherical in
$G \times G$.

Note that the requirement of having an open orbit in $G/B$ implies
that $H$ has, in fact, only finitely many orbits on $G/B$ (see, e.g.,
\cite{Knop3}). Therefore, our classification theorem can also be
viewed as a contribution to the program by Seitz \cite{Seitz} to
classify all pairs of subgroups $X,Y$ of a reductive group $G$ such
that there are only finitely many $(X,Y)$-double cosets in $G$ (see
also \cite{Br}).

The most important previous works are the aforementioned
classifications due to Mikityuk \cite[\S 5]{mik} and Brion
\cite{Bri87}.  Not only do we use their lists as a guideline, but, more
importantly, they enter crucially into our computations even in positive
characteristic.  This is because we extensively employ the technique
of reduction mod $p$ developed in \cite{KR}.  Here, we show that all
instances from the Brion--Mikityuk list descend to arbitrary positive
characteristic.

In \cite{KR}, we have already observed that for semisimple groups $G$
in positive characteristic, infinitely many ``new'' spherical subgroups
can arise in connection with a finite orbit module involving Frobenius
twists, cf.~\cite[Lem.\ 2.6]{GLMS}.  Specifically, consider the
irreducible representation $\Delta_q:\SL(2)\to\GL(4)$ of $\SL(2)$ with
highest weight $(q+1)\omega_1$ with $q=p^m>1$.  Since $\Delta_q$ is
self-dual, its image lies in $\SOO(4)$.

Note that for the purpose of classifying spherical subgroups, we may
replace $G$ with an isogenous group (using
\Cref{lem:isogeny}). Therefore, the simply connected Spin groups do
not appear in Table~\ref{table:sp-classical}, for
instance, but rather their isogenous counterparts do.

As the case when $G$ is simple is handled in \cite{KR}, in \Cref{thm:main} we restrict ourselves to the non-simple case. Thanks to \Cref{lem:ss}, we can reduce to the instance when $G$ is semisimple. For the concepts of an indecomposable spherical pair and for an isogeny between spherical pairs, see \Cref{def:pairs}.

The only novelty is that, up to
isogeny, there is essentially only one series (case (S9) in
\Cref{Tab:BMKR}) that is genuinely unique to positive
characteristic, namely in characteristic $2$.

\begin{theorem}
  \label{thm:main}
  Let $G$ be a connected semisimple algebraic group and let
  $H \subset G$ be a spherical connected reductive subgroup of
  $G$. Assume that $G$ is not simple and that the pair $(G,H)$ is
  indecomposable. Then $(G,H)$ is isogenous to one of the items of
  \Cref{Tab:BMKR}.
\end{theorem}

Here, ``isogeny'' does not only mean only central isogenies but also outer automorphisms, Frobenius morphisms, and exceptional
isogenies. Taking this into account, in positive characteristic a
given group $G$ might contain an infinitude of spherical connected
reductive subgroups.


\textbf{Acknowledgments:} Work on this paper began during a visit to the Mathematisches Forschungsinstitut Oberwolfach (MFO) under the ``Oberwolfach Research Fellows'' program; we thank the MFO for its support.

\section{Preliminaries}
\label{sec:basics}

\subsection{Notation}
Throughout, $G$ is a connected reductive algebraic group defined over
an algebraically closed field $k$ of characteristic $p\geqslant  0$. Let $B$ be a
Borel subgroup of $G$. By $\rk G$ we denote the rank of $G$.

Unless stated otherwise, all subgroups are assumed to be closed
and reduced.  Let $H\subseteq G$ be a subgroup. Then $H^0$ denotes the identity component of $H$. Moreover, $\Delta H \subset H \times H$ is the diagonal embedding of $H$ into $H \times H$. If $H$ acts
on the variety $X$, we write  $Hx$ for the $H$-orbit of $x\in X$  and $C_H(x)$ for its (reduced) stabilizer in $H$.

A homomorphism $\phi:H_1\to H_2$ between connected groups is an
\emph{isogeny} if it is surjective with finite kernel. We say that
$H_1$ and $H_2$ are isogenous and denote this by $H_1 \sim H_2$, if
there is a connected group $H_0$ and isogenies
$H_1\leftarrow H_0\to H_2$. More generally:

\begin{definition}
  \label{def:pairs}
  Let $G$ and $H$ be connected and reductive.
  \begin{itemize}
  	\item [(i)] A \emph{homogeneous pair} $(G,H;\rho)$ is given by a homomorphism $\rho:H\to G$ such that $H\to\rho(H)$ is an isogeny. We usually suppress $\rho$ if no
  	confusion can arise. The homogeneous variety $G/\rho(H)$ is also
  	denoted by $G/H$. The pair $(G,H)$ is \emph{trivial} if
  	$\rho(H)=G$.
  	\item [(ii)] An \emph{isogeny between two homogeneous pairs}
  	$(G_1,H_1;\rho_1)$ and $(G_2,H_2;\rho_2)$ is an isogeny
  	$\phi:G_1\to G_2$ which maps $\rho_1(H_1)$ onto $\rho_2(H_2)$. This
  	induces a surjective finite morphism $G_1/H_1\to G_2/H_2$. Two
  	homogeneous pairs are said to be \emph{isogenous} if there is a
  	homogeneous pair $(G_0,H_0)$ and isogenies
  	$(G_1,H_1)\leftarrow(G_0,H_0)\to(G_2,H_2)$.
  	\item [(iii)] A pair $(G,H)$ is \emph{decomposable} if it is isogenous
  	to a product pair $(G_1\times G_2,H_1\times H_2)$ where both pairs
  	$(G_1,H_1)$ and $(G_2,H_2)$ are non-trivial. A pair which is neither
  	trivial nor decomposable is called \emph{indecomposable}.
  \end{itemize}
\end{definition}

\begin{remark}
  The advantage of this notion is that it allows us to assume that both
  $G$ and $H$ are the Cartesian product of a torus and simple factors.
\end{remark}

\subsection{On some subgroups of simple groups}

The following auxiliary results are used in the proof of \Cref{lem:3}
below.

\begin{lemma}
  \label{lem:diag1}
  Let $G$ be simple and let $H \subsetneq G$ be a proper subgroup of
  maximal dimension. Then $H$ is a parabolic subgroup of $G$.
\end{lemma}

\begin{proof}
  We claim that there is a finite-dimensional $G$-module $V$ such that
  $\mathbb{P}(V)^G = \leer$ and $\mathbb{P}(V)^H \ne \leer$.  Thanks
  to Chevalley's Theorem, see \cite[Ch.~II, Thm.~5.1]{borel}, there is
  a finite-dimensional $G$-module $\tilde V$ and
  $y \in \mathbb{P}(\tilde V)$ such that $C_G(y) = H$.  Let
  $V: = \tilde{V}/\tilde{V}^G$.  Since $\Ext_G^1(k,k) = 0$, e.g.,
  see \cite[(2.12(1))]{jantzen}, we have $V^G = 0$. But then
  $\mathbb{P}(V)^G = \leer$, since the character group of $G$ is
  trivial. Now observe that $y \notin \mathbb{P}(\tilde V^G)$, since
  $C_G(y) = H \subsetneq G$. Therefore, the image of $y$ in
  $\mathbb{P}(V)$ is an $H$-fixed point, proving the claim.
	
  Now let $V$ be as in the claim and choose $y\in \mathbb{P}(V)^H$.
  Then the orbit map induces a map
  $\phi : X : = G/H \to \mathbb{P}(V)$ mapping $x := 1H\in X$ to $y$.
  If $\dim \phi (X) < \dim X$, then $\dim C_G(y) > \dim H$,
  contradicting the maximality of $\dim H$.  Otherwise,
  $\dim \phi (X) = \dim X$. In this case, $C_G(y)$ contains
  $H = C_G(x)$ as a subgroup of finite index. Suppose $H$ is not
  parabolic in $G$. Then $C_G(y)$ is not parabolic either. Hence
  $\phi (X)$ is not complete and therefore not closed in
  $\mathbb{P}(V)$.  Pick
  $z\in Z : = \overline {\phi(X)} \setminus \phi(X)$ in
  $\mathbb{P}(V)$.  Then
  $\dim Gz \leqslant \dim Z < \dim \phi(X) = \dim Gx$. Therefore,
  $\dim C_G(z) > \dim C_G(x) = \dim H$, again contradicting the
  maximality of $\dim H$. Thus $H$ is parabolic after all.
\end{proof}

\begin{lemma}
  \label{lem:4}
  Let $G$ be simple and let $H\subsetneq G$ be a proper subgroup. Then
  $\dim G/H \geqslant \rk G$.
\end{lemma}

\begin{proof}
  Let $T \subseteq G$ be a maximal torus.  By the rigidity of tori,
  there exists a subgroup $S$ of $T$ such that $S = C_T(x)$ for a
  generic $x \in G/H$. Then $S$ acts trivially on $G/H$ and so $S$
  lies in the kernel of the action of $G$ on $G/H$. Since $G$ is
  simple and $H \subsetneq G$, it follows that $S^0 =
  1$. Consequently, there is an $x$ in $G/H$ such that
  $\dim G/H \geqslant \dim Tx = \dim T = \rk G$, and the result follows.
\end{proof}

\subsection{Reduction results for spherical subgroups}

The first well-known observation, which is immediate from the
definition, allows us to reduce to the case when $H$ is connected.

\begin{lemma}
  \label{lemma0}
  Let $H\subseteq G$ be a subgroup of $G$.  Then $H$ is spherical in
  $G$ if and only if the identity component $H^0$ is spherical.
\end{lemma}

The next observation allows us to reduce to the case when $G$ is
semisimple, cf.~\cite[\S 4]{Kr}, which we assume henceforth.

\begin{lemma}
  \label{lem:ss}
  Let $G$ be reductive and let $Z$ be its connected center. Then $H$
  is spherical in $G$ if and only if $ZH/Z$ is spherical in $G/Z$.
\end{lemma}

The next lemma states that the set of spherical subgroups of $G$
depends only on its isogeny class.

\begin{lemma}
  \label{lem:isogeny}
  Let $\phi:G_1\to G_2$ be an isogeny between connected reductive
  groups.  Then $\phi$ induces a bijection between the sets of
  (conjugacy classes of) connected (reductive) spherical subgroups of
  $G_1$ and $G_2$.
\end{lemma}

\begin{definition}
	\label{def:sphpair}
	A homogeneous pair $(G,H)$ is called \emph{spherical} if $G/H$ is	spherical. 
\end{definition}

The lemma implies that this property is
isogeny-invariant. It is clear that a product of homogeneous pairs is
spherical if and only if each factor is spherical. Thus, we may
restrict our attention to indecomposable spherical pairs.

\subsection{Deformation of spherical subgroups}
\label{subsec:deformation}
In this subsection, we recall some results from \cite[Thm.\ 3.4]{KR}
which enables us to compare spherical subgroups in positive
characteristic to those in characteristic zero. This approach reduces
most of the classification work to the results of Brion--Mikityuk.

Let $\CG$ be a split reductive group scheme over $\BBZ$ (this entails
connected geometric fibers), e.g., see \cite[Exp.~I,
4.2]{demazurgrothendieck:schemasengroupes}, and $\CH\subseteq\CG$ a
closed flat subgroup scheme. Then for any algebraically closed field
$k$, we obtain the geometric fibers $H_k\subseteq G_k$ by base change
and the homogeneous variety $X_k=G_k/H_k$. We say that a homogeneous pair
$(G,H)$ defined over $k$ \emph{lifts to characteristic zero} if it is
isogenous to $X_k$ for some pair $(\CG,\CH)$. In this case, the lift
is $X_\BBC$, where $\BBC$ is the field of complex numbers. If
$k=\BBC$, we also say that $X_\BBC$ is defined over $\BBZ$.  Now it
follows from \cite[Thm.\ 3.4]{KR}:

\begin{lemma}
  \label{Theorem1}
  Assume that the homogeneous $G$-variety $X$ defined over $k$
  lifts to characteristic zero. Then $X$ is spherical if and only if its
  lift $X_\BBC$ is spherical.
\end{lemma}

\begin{remark}
  An inspection of the tables by Brion and Mikityuk shows that all
  indecomposable spherical pairs $(G,H)$ defined over $\BBC$ with $G$
  not simple have a $\BBZ$-form. Their $k$-forms comprise items
  (S0)--(S8) of \Cref{Tab:BMKR}.
\end{remark}

\subsection{Some sphericity criteria}

The following numerical criterion is fundamental in identifying or
ruling out candidates for spherical subgroups.

\begin{lemma}
  \label{lemma3}
  Let $H\subseteq G$ be spherical in $G$. Then
  \begin{equation}
    \label{eq:2}
    \dim H\geqslant\dim G/B=\frac{1}{2}(\dim G-\rk G).
  \end{equation}
\end{lemma}

\begin{proof}
  By definition, $B$ has an open orbit in $G/H$. Hence
  $\dim B\geqslant \dim G/H$ which is equivalent to \eqref{eq:2}.
\end{proof}

We frequently use the following ``transitivity'' property for
spherical subgroups, e.g., see \cite[Lem.~2.3]{KR}.

\begin{lemma}[Weak Transitivity]\label{lemma:weak}
  \label{lem:subgroups}
  Let $H_1\subseteq H_2\subseteq G$ be connected reductive subgroups
  of $G$.  If $H_1$ is spherical in $G$, then $H_1$ is spherical in
  $H_2$ and $H_2$ is spherical in $G$.
\end{lemma}

This criterion is just necessary. To actually show that $H_1$ is
spherical in $G$ we need more detailed information. To this end,
consider for any quasi-affine $G$-variety $X$ the group $\Xi(X)$ of
characters of $B$-semi-invariant rational functions on $X$. We define
the rank of $X$ as the $\BBZ$-rank of $\Xi(X)$. Note that $\Xi(X)$ is
generated by the set of highest weights occurring in the
representation $k[X]$ of regular functions on $X$.

Let $S_0$ be the set of simple roots $\alpha$ of $G$ (with respect to
$B$) such that the coroot $\alpha^\vee$ is orthogonal to
$\Xi(X)$. Then attached to $S_0$, there is a parabolic subgroup
$P = P(X)$ of $G$ such that $\Xi(X)\subseteq\Xi(P)$, where $\Xi(P)$ is
the character group of $P$.  We define the subgroup $P_0$ of $P$ by
$P_0 = \{y \in P \mid \chi(y)=1\text{ for all } \chi\in\Xi(X)\}$.

\begin{theorem}[{\cite[Thm.\ 4.1]{KR}}]
  \label{thm:genstab}
  Let $X$ be a quasi-affine $G$-variety. Let $P = P(X)$ be as
  above. Then there is a $P$-invariant dense open subset $X_0$ of $X$
  such that $C_P(x) R_u(P) = P_0$ and $C_P(x)\cap R_u(P)$ is finite
  for all $x\in X_0$. In particular, $C_P(x)$ is a reductive group
  which is isogenous to a Levi subgroup of $P_0$.
\end{theorem}

\begin{definition}\label{def:LX}
  In the setting of \Cref{thm:genstab}, let $L(X):=C_P(x)^0$ for any $x\in X_0$. It is
  isogenous to a Levi factor of $P_0^0$.
\end{definition}

Observe that if we put $H:=C_G(x)$, then $L(X)=(P\cap H)^0$ can also be
considered as a subgroup of $H$. The significance of $L(X)$ is a
converse to \Cref{lemma:weak}:

\begin{lemma}[Strong Transitivity]\label{lemma:hard}
  Let $G$ be a connected, reductive group and let
  $H_1\subseteq H_2\subseteq G$ be connected, reductive subgroups. Put
  $L:=L(G/H_2)$. Then $G/H_1$ is spherical if and only if $G/H_2$ is
  spherical and $H_2/H_1$ is spherical as an $L$-variety.
\end{lemma}

\begin{proof}
  Consider the canonical morphism $\phi:X_1:=G/H_1\to X_2:=G/H_2$. Let $x$ be a
  point in the open $B$-orbit $X_2^0$ of $X_2$. Since $X_2^0$ is also
  $P$-invariant, we get a bijective morphism $P/C_P(x)\to X_2^0$, hence a
  bijective morphism
  $P\times^{C_P(x)}F\to\phi^{-1}(X_2^0)\subseteq X_1$ with
  $F :=\phi^{-1}(x)\cong H_2/H_1$. Thus, $B$ has an open orbit in $X_1$
  if and only if $C_P(x)=C_B(x)$ has an open orbit in $F$. Now
  $C_B(x)$ is a Borel subgroup of $C_P(x)$. Thus $X_1$ is spherical for $G$ if
  and only if $F$ is spherical as an $L (=C_P(x)^0)$-variety.
\end{proof}

To compute $L(X)$ we use the following comparison result. For this put
$\Xi_\BBQ(X):=\Xi(X)\otimes_\BBZ\BBQ$.

 \begin{lemma}[{\cite[Lem.~3.7]{Knop5}}]
   Assume that $X_\BBC$ has a $\BBZ$-form with geometric fiber $X_k$.  Then
   $\Xi_\BBQ(X_k)=\Xi_\BBQ(X_\BBC)$. In particular, the group $L(X_k)$
   is a reductive group whose root datum is, up to isogeny, the same
   sub-root system of the root system of $G$ as that of $L(X_\BBC)$.
 \end{lemma}

\begin{remark}
  Over $\BBC$, Bravi--Pezzini calculated in \cite{BrPezz} the Luna
  diagrams of all affine, homogeneous, spherical varieties. The group
  $L(G/H)$ can easily be read off: First, the simple roots of $P(X)$
  are those without ``decoration''. Second, the rank of the character
  group $\Xi(X)$ equals the number of spherical roots. From this,
  $L(G/H)$ can be determined up to isogeny. It is indicated in the
  fourth column of \Cref{Tab:BMKR}.
\end{remark}

\section{Soundness of \Cref{Tab:BMKR}}

\begin{lemma}
  \label{lem:known}
  All varieties $G/H$ from \Cref{Tab:BMKR} are spherical.
\end{lemma}

\begin{proof}
  All items of \Cref{Tab:BMKR}, except for (S9), are defined over
  $\BBZ$. Hence their sphericity follows via \Cref{Theorem1} from
  the work of Mikityuk and Brion.
	
  It remains to show that the case (S9) is spherical: Let $\bar H$ be the
  subgroup $(\G_2 \times \Sp(2)) \times (\Sp(2) \times \Sp(2n))$ of
  $G$ such that $H$ is contained in $\bar H$ diagonally.  Note that by
  \Cref{lemma:hard}, $G/H$ is spherical if and only if $G/ \bar H$ is
  spherical and $\bar H/H$ is spherical as an $L(G/\bar H)$-variety.  Now
  $G/\bar H = X_1 \times X_2$, where
  $X_1 = \Sp(8)/(\G_2 \times \Sp(2))$ and
  $X_2 = \Sp(2n +2 )/(\Sp(2) \times \Sp(2n))$.  Since both $X_1$ and
  $X_2$ are spherical by \cite{KR}, $G/\bar H$ is also spherical.
	
  Now $L(G/\bar H)$ decomposes as $L_1 \times L_2$, where
  $L_1 \cong \mathbb{G}_m$, see \cite[Prop.\ 4.5]{KR}, and
  $L_2 \cong \Sp(2) \times \Sp(2n-2)$. Since $\bar H/H \cong \Sp(2)$,
  we see that $G/H$ is spherical if and only if $\Sp(2)$ is spherical
  as an $(S \times \Sp(2))$-variety, where $S$ is the image of $L_1$
  in $\Sp(2)$. It remains to show that $S \ne 1$.
	
  We have
  $\Xi(X_1)=\langle \omega_1+\omega_4,\omega_2,\omega_3\rangle_\BBZ$,
  thanks to \cite[Prop.\ 4.5]{KR}, where $\omega_i$ denotes the $i$-th
  fundamental weight of $G$ with respect to the labeling in
  \cite[Planche III]{bourbaki:groupes}.  Hence $L_1$ is the
  $1$-parameter subgroup of the Borel subgroup of $\Sp(8)$
  corresponding to the cocharacter
  $\lambda:=\alpha_1^\vee - \alpha_4^\vee = \epsilon_1 - \epsilon_2 -
  \epsilon_4$. On the other hand, the characters of the maximal torus
  of $\G_2$ inside $\Sp(6)$ are those which are
  orthogonal to $\epsilon_1 + \epsilon_2 + \epsilon_3$ (see
  \cite[Planche IX]{bourbaki:groupes}). Now $L_1$ as a subgroup of $H$ is
  $W_{\Sp(8)}$-conjugate to $\lambda$. Hence it is of the form
  $\pm \epsilon_i \pm \epsilon_j \pm \epsilon_k$ for
  $1 \leqslant i < j < k \leqslant 4$. Being orthogonal to
  $\epsilon_1 + \epsilon_2 + \epsilon_3$ implies $k = 4$.  This
  entails that the projection $S$ of $L_1$ to the $\Sp(2)$-factor of
  $\G_2 \times \Sp(2)$ is non-trivial.
\end{proof}

\section{Completeness of \Cref{Tab:BMKR}}

\begin{definition}
  \label{def:knownpair}
  We say that a pair $(G,H)$ is \emph{known spherical} if it is
  isogenous to a product of trivial pairs, pairs contained in
  \cite[Tables~1 and~2]{KR}, or pairs from \Cref{Tab:BMKR}.
\end{definition}

Since Table 1 of \cite{KR} is used extensively, we restate
it here in a condensed form for the convenience of the reader. More
precisely, the table is complete only \emph{up to isogeny}, i.e., up
to conjugation, outer automorphisms (such as triality), central
isogenies, and totally inseparable isogenies (such as Frobenius morphisms).

\setcounter{myrow}{0}

\begin{table}[h!t]
\caption{Spherical pairs $H\subset G$ with $G$ classical up to isogeny.} 
\label{table:sp-classical} 
\[
  \begin{array}{>{\stepcounter{myrow}\text{(C\themyrow)}}llll}
\hline
\multicolumn{1}{l}{}& H & G\\
\hline
&\SOO(n) & \SL(n)&n\geqslant 2\\
&S(\GL(m){\times}\GL(n)) & \SL(m{+}n)&m\geqslant n\geqslant 1\\
&\SL(m)\times\SL(n)&\SL(m{+}n)&m>n\geqslant 1\\
&\Sp(2n)&\SL(2n)&n\geqslant 2\\
&\Gm\cdot\Sp(2n)&\SL(2n+1)&n\geqslant 1\\
&\Sp(2n)&\SL(2n+1)&n\geqslant 1\\
\hline
&\GL(n)&\Sp(2n)&n\geqslant 1\\
&\Gm\times\Sp(2n-2)&\Sp(2n)&n\geqslant2\\
&\Sp(2m){\times}\Sp(2n)&\Sp(2m{+}2n)&m, n\geqslant 1\\
\hline                                
&\GL(n)&\SOO(2n)&n\geqslant 2\\
&\SL(n)&\SOO(2n)&n\geqslant 3\text{ odd}\\
&\G_2&\SOO(8)\\
&\SOO(2)\times\Spin(7)&\SOO(10)\\
&\GL(n)&\SOO(2n+1)&n\geqslant 2\\
&\SOO(m){\times}\SOO(n)&\SOO(m{+}n)&m\geqslant n\geqslant 1\\
&\Spin(7)&\SOO(9)\\
&\G_2&\SOO(7)\\
\hline
&\G_2\times\Sp(2)&\Sp(8)&{\color{red}p=2}\\
\hline                                
\end{array}  
\]
\end{table}

The proof of \Cref{thm:main} reduces easily to the following
statement.

\begin{lemma}
  \label{lem:1}
  Suppose that $(G,H)$ is known spherical and $H'\subset H$ is a maximal,
  connected reductive subgroup. Then $(G,H')$ is known spherical or
  not spherical.
\end{lemma}

\begin{proof}[Proof of \Cref{thm:main}] 
  In view of \Cref{lem:known}, it suffices to show that every spherical
  pair $(G,H)$ is known spherical. If $H=G$, then this is clear by
  definition. Otherwise, there is a connected reductive intermediate
  subgroup $H \subset \bar H \subseteq G$ such that $H$ is a maximal
  connected reductive subgroup of $\bar H$. Then $G/\bar H$ is
  spherical by \Cref{lem:subgroups}. Thus, $(G,H)$ is known spherical
  by \Cref{lem:1}.
\end{proof}

The vast majority of cases to be inspected are handled by the following
lemma. It utilizes the classification results of Krämer, Mikityuk and Brion
to their full extent.

\begin{lemma}
  \label{lem:2}
  Assume that the triple $(G,H,H')$ in \Cref{lem:1} is defined over
  $\BBZ$. Then the assertion of \Cref{lem:1} holds for $(G,H,H')$.
\end{lemma}

\begin{proof}
  Assume that the pair $(G,H')$ is spherical. Then, by \Cref{Theorem1}, its
  lift to characteristic zero is spherical as well. Hence, this lift is
  known spherical by the work of Krämer, Mikityuk, and Brion. Since the
  reductions modulo $p$ of their tables are contained in our tables, $(G,H')$
  is also known spherical.
\end{proof}

\begin{remark}
  \label{rem:1}
  Observe that every indecomposable factor of a known spherical pair
  is isogenous to one that lifts to characteristic zero, except for
  the cases isogenous to case (C18) in
  \Cref{table:sp-classical} or case (S9) in \Cref{Tab:BMKR}.
\end{remark}

Let $H$ be a connected reductive group and let $H'\subset H$ be a
maximal connected reductive subgroup. Extending the version of Dynkin's Lemma
\cite[Lem.~3.3]{Knop5} slightly from semisimple to
reductive groups, we obtain the following cases where $H'$ is obtained from
$H$
\begin{enumerate}[\text{(\alph{enumi})}]
\item\label{Max1} by replacing the connected center of $H$ with a subtorus of
  codimension one, or
\item \label{Max2} by replacing a simple factor of $H$ with a maximal
  connected reductive subgroup, or
\item \label{Max3} by replacing two simple factors $H_1$ and $H_2$ of $H$ with a
  connected subgroup $\tilde H\subset H_1\times H_2$ for which both
  projections $\tilde H\to H_i/(H_1\cap H_2)$ are isogenies.
\end{enumerate}

\begin{proof}[Proof of \Cref{lem:1} in case \ref{Max1}]
  By \Cref{rem:1}, we may assume that all factors of $(G,H)$ lift to
  characteristic zero, except for those of type (C18) or (S9). Since
  these exceptions do not contribute to $Z(H)^0$, we may assume that they do not
  occur.  Thus, $(G,H)$ lifts to characteristic zero. The torus
  $Z(H)^0$ lifts to characteristic zero as well, along with all of its
  subgroups. Thus, $(G,H,H')$ lifts to characteristic zero, and the
  assertion follows from \Cref{lem:2}.
\end{proof}

\emph{Proof of \Cref{lem:1} in case \ref{Max2}}.  Assume that $(G,H')$ is
spherical.  We may further assume that $(G,H)$ is indecomposable, as
only one indecomposable factor of $(G,H)$ is modified in case \ref{Max2}.  If
$G$ is simple, then $(G,H')$ appears in \cite[Tables 1 and 2]{KR} and
is thus known spherical. If $G$ is not simple, then
$(G,H)$ is isogenous to one of the items in \Cref{Tab:BMKR}. We now
analyze these cases individually.

(S0) is dealt with by the following lemma.

\begin{lemma}
  \label{lem:3}
  Let $G$ be simple and let $H \subset G$ be a maximal connected
  reductive subgroup. Then $(G\times G,\Delta H)$ is not spherical.
\end{lemma}

\begin{proof}
  We have $\Delta H \subset \Delta G \subset G \times G$. Since
  $L(G \times G/\Delta G) = T$, where $T$ is a maximal torus of $G$
  (see \Cref{Tab:BMKR}), it suffices to show that
  $\Delta G/\Delta H \cong G/H$ is not spherical as a $T$-variety  (see
  \Cref{lemma:hard}). By \Cref{lem:diag1}, $G$
  contains a proper subgroup $H'$ such that $\dim G/H > \dim
  G/H'$. Moreover, by \Cref{lem:4}, we have $\dim G/H' \geqslant \dim
  T$. Thus, $G/H$ is not spherical  as a $T$-variety, as claimed.
\end{proof}

For the cases where $(G,H)$ is of type (S1) or (S2), we have
$L(G/H)=1$ by \Cref{Tab:BMKR}. Thus, the pair $(G,H')$ is not
spherical by \Cref{lemma:hard}.

Let $(G,H)$ be one of the pairs (S3) -- (S8).  Since $H/H'$ is also
spherical (\Cref{lemma:weak}), it occurs in
\Cref{table:sp-classical}.  Hence, for $p \ne 2$, the pair $(H,H')$ can be
lifted to characteristic zero. Consequently, $(G,H,H')$ can also be lifted, and
the assertion follows from \Cref{lem:2}.

Now suppose $p =2$. Then we require the following lemma.

\begin{lemma}
  \label{lem:p21}
  Let $p=2$ and $n\geqslant 2$. Then $\Sp(2n)/\SOO(2n)$ is not spherical as
  a $(\Gm \times \Sp(2n-2))$-variety. In particular, it is
  not spherical as an $\Sp(2n-4)$-variety. Moreover, $\Sp(6)/\G_2$ is
  not spherical as an $\Sp(4)$-variety.
\end{lemma}

\begin{proof}
  Let $(x_1, \ldots, x_{2n})$ be the standard basis of $k^{2n}$. Then
  $S^2k^{2n}$ is the space of quadratic forms in the variables $x_i$.  Since
  $p = 2$, we have the short exact sequence:
  \[
    \begin{tikzcd} 0 \arrow{r} & (k^{2n})^{(2)} \arrow{r} & S^2k^{2n}
      \arrow["\pi"]{r} & \bigwedge\nolimits^2 k^{2n} \arrow{r} &
      0,\end{tikzcd}
  \]
  where we identify $(k^{2n})^{(2)}$ with the subspace spanned by
  $x_i^2$ for $1 \leqslant i \leqslant 2n$ and $\pi$ is the polarization map
  (e.g., see \cite[Def.~7.1]{EKM}).  Let
  $q := x_1 x_{2n} + \ldots +x_n x_{n+1}$.  Then
  $\omega := \pi (q) = x_1 \wedge x_{2n} + \ldots +x_n \wedge x_{n+1}$
  defines the group $\Sp(2n)$.  Let
  $X := \pi^{-1}(\omega) = q + \langle x_1^2, \ldots , x_{2n}^2
  \rangle \cong \mathbb{A}^{2n}$. Then $X$ can be identified with
  $\Sp(2n)/\O(2n)$. On the other hand, the map
  $X \to S^2 k^2 = S^2 \langle x_1 , x_{2n} \rangle$ is
  $\Sp(2n-2)$-invariant and non-constant. This shows that the quotient
  $X\mod\Sp(2n-2)$ is at least $2$-dimensional. Consequently, not even
  $\mathbb{G}_m \times \Sp(2n-2)$ has a dense orbit in $X$, and therefore
  neither does its Borel subgroup. The same holds for the double cover
  $\Sp(2n)/\SOO(2n)$.  The last statement follows from the dimension
  criterion in \Cref{lemma3}.
\end{proof}

We now return to the case $H'\subset H\subset G$ for $p=2$. Note that
all connected subgroups of $\Sp(2)$ still lift to characteristic
zero. Thus, we can argue as before.  Therefore only the additional
cases $H/H' \sim \Sp(2n)/\SOO(2n)\sim\SOO(n+1)/\SOO(n)$ for $n \geqslant 2$
and $H/H'\sim\Sp(6)/\G_2$, which do not lift, need to be considered.
Observe that, by \Cref{lem:p21}, $H/H'$ is not spherical as an
$L(G/H)$-variety in each instance with the exception of case (S6) when
$H/H' \cong \Sp(6)/\G_2$.  In this case replacing the $\Sp(6) $-factor
with $\G_2$ results in a spherical pair, namely (S9). Hence, in all
cases, $(G,H')$ is either not spherical or known spherical.
 
Finally, let $(G,H)$ be of type (S9), and let $H_0$ be a maximal
reductive subgroup of $\G_2$. Let
$H' = H_0 \times \Sp(2) \times \Sp(2n)$. If $G/H'$ were spherical,
then the quotient $\Sp(8)/(H_0 \times \Sp(2))$ would also be
spherical. However, this contradicts the dimension bound \eqref{eq:2}.
The same argument works for the $\Sp(2)$-factor.

If, on the other hand, the $\Sp(2n)$-factor is replaced with a maximal
reductive subgroup $H_0$, then $H'$ is contained in
$\Sp(6) \times \Sp(2) \times \Sp(2n)$.  The discussion of case
  (S7) above shows that $\Sp(6) \times \Sp(2) \times H_0$ is
spherical if and only if $H_0 = \G_2$. However,
$\G_2 \times \Sp(2) \times \G_2$ inside $\Sp(8) \times \Sp(8)$ is not
spherical by the dimension bound \eqref{eq:2}.  This concludes the
proof of \Cref{lem:1} in case \ref{Max2}.
	
\emph{Proof of \Cref{lem:1} in case \ref{Max3}.} Let $H_1'$ and $H_2'$ be two
different simple factors of $H$, and for $i = 1,2$, let
$\rho_i:\tilde H\to H_i'$ be isogenies. Then, $H'$ is obtained
by replacing $H_1'H_2'\subseteq H$ with the image of
$\rho_1\cdot\rho_2$. There are now two cases to consider, namely:
\begin{enumerate} [\text{(\roman*)}]
\item\label{it:A1} the subgroups $H_i'$ lie in the same indecomposable component of $(G,H)$,
  and
\item\label{it:A2} the subgroups $H_i'$ lie in two different components of
  $(G,H)$.
\end{enumerate}
Note also that the isogenies $\rho_i$ might not be central, which could
prevent liftability to characteristic zero.

We claim that in case \ref{it:A1}, the pair $(G,H')$ is not spherical. If $G$ is
simple, the cases where $H$ has two isogenous factors are:
$(\SL(4),\SOO(4))$, $(\SL(2n),\SL(n)\Gm\SL(n))$ for $n\geqslant 2$,
$(\Sp(4n),\Sp(2n)\Sp(2n))$ for $n\geqslant 1$, $(\SOO(2n),\SOO(n)\SOO(n))$
for $n\geqslant 3$, and $(\SOO(n), \SOO(4)\SOO(n-4))$.  None of these cases lead to
$(G,H')$ being spherical, as can be seen by inspecting \Cref{table:sp-classical}.

If $G$ is not simple, then $(G,H)$ belongs to \Cref{Tab:BMKR}. If $H_1'$ and
$H_2'$ are subgroups of the same simple factor of $G$, then the
projection of $(G,H')$ is not spherical by the case where $G$ is simple. Thus, we
must still check case (S4) with $(m,n)=(2,1)$ (which also settles (S5)), case (S6)
with $m=n$ (which also settles (S7) with the first and
third factors of $H$ identified and (S8)), and case (S7) with $m=n$. By
\Cref{lemma:hard}, we need to show that $H_0\sim H_1'H_2'/H_0$ is not
spherical as an $L(G/H)$-variety. This results in an action of
$\Sp(2n-2)\times\Sp(2n-2)$ on $\Sp(2n)$, which is not spherical
by \Cref{lemma3}. This concludes the proof in case \ref{it:A1}.

In case \ref{it:A2}, let $(G,H)=(G_1,H_1)\times(G_2,H_2)$ be such that a
factor $H_0$ of $H'$ is embedded diagonally in $H_1\times H_2$. First,
we treat the case where $(G_2,H_2)$ is a trivial pair, i.e.,
$G_2=H_2$. Then, the diagram\footnote{Here, and in
  \Cref{Tab:BMKR} in particular, this graphical description of a pair $(G,H)$ generalizes the notation from \cite{mik}: nodes symbolize simple (or
  trivial) factors, and edges indicate isogenies onto their images.} of
$(G,H')$ looks like
\begin{equation}\label{eq:G1G2H0H1H2}
  \begin{tikzpicture}[baseline=(current bounding box.center), scale=0.6, every node/.style={inner sep=2pt}]
    \draw (0,0) -- (1,1); \draw (1,1) -- (2,0); \draw (2,0) -- (3,1);
    \fill (0,0) circle (0.1) node[below] {$\scriptstyle
      H_0'$}; \fill (1,1) circle (0.1) node[above]
    {$\scriptstyle
      G_1$}; \fill (2,0) circle (0.1) node[below]
    {$\scriptstyle
      H_0$}; \fill (3,1) circle (0.1) node[above]
    {$\scriptstyle
      G_2$}; \fill (1.3,0.5) circle (0) node[below]
    {$\scriptstyle
      \rho_1$}; \fill (2.7,0.5) circle (0) node[below]
    {$\scriptstyle \rho_2$};
  \end{tikzpicture}
\end{equation}
where $H_0'$ is a connected reductive group (possibly trivial),
$\rho_1$ has a finite kernel, and $\rho_2$ is an isogeny. Let $\tilde H$
be the image of $H_0$ in $G_1\times G_2$, and let $p_i:\tilde H\to G_i$
be the projection for $i = 1,2$. Then,
$(\id,p_2):G_1\times \tilde H\to G_1\times G_2$ is an isogeny that
maps $(p_1,\id)(\tilde H)$ to $(\rho_1,\rho_2)(H_0)$. On the other
hand, the isogeny
$(\id,p_1):G_1\times \tilde H\to G_1\times\rho_1(H_0)$ maps
$(p_1,\id)(\tilde H)$ to $(\id,\id)(\rho_1(H_0))$. Thus, we have
shown that $(G,H')$ is isogenous to a pair where $\rho_1$ is an
inclusion and $\rho_2$ is the identity. In particular, this implies
that if $(G_1,H_0'\times H_0)$ lifts to characteristic zero, then
$(G,H')$ also does. Therefore, the assertion of \Cref{lem:1} is
proved in all cases except when $(G_1,H_1)$ is of type
(C18), which cannot be lifted to characteristic zero. If $H_0=\Sp(2)$,
we obtain case (S9) with $n=2$ which is therefore known spherical. If, on the other hand, $H_0=\G_2$,
we obtain the pair $(\Sp(8)\times\G_2,\Sp(2)\times\G_2)$, which is not
spherical by the dimension bound \eqref{eq:2}. Thus, we are done with
pairs of type \eqref{eq:G1G2H0H1H2}.

The analysis of pairs of type \eqref{eq:G1G2H0H1H2} also shows which
factors can occur as $(G_1,H_1)$. In fact, these are the ones fitting
into the diagram \eqref{eq:G1G2H0H1H2} that occur in
\Cref{Tab:BMKR}. Thus, we obtain the following cases:
\begin{center}
  \begin{tabular}{l|l|l|l|l}
    Type&$G$&$H$&&$L'$\\
    \hline
    (S$0'$)&$H$&$\ul{H}$&$H$ simple&T\\
    (S$1'$)&$\SL(n+1)$&$\Gm\,\ul{\SL(n)}$&$n\geqslant 3$&$\GL(n-1)$\\
    (S$2'$)&$\SOO(n+1)$&$\ul{\SOO(n)}$&$n\geqslant 5$&$\SOO(n-1)$\\
    (S$3'$)&$\Sp(2n+4)$&$\Sp(2n)\,\ul{\Sp(4)}$&$n \geqslant 1$&$\Sp(2)\Sp(2)$\\
    (S$4'$)&$\SL(n+2)$&$\Gm\SL(n)\,\ul{\Sp(2)}$&$n \geqslant 1$&$\Gm$\\
    (S$5'$)&$\SL(n+2)$&$\SL(n)\,\ul{\Sp(2)}$&$n \geqslant 3$&$\Gm$\\
    (S$6'$)&$\Sp(2n+2)$&$\Sp(2n)\,\ul{\Sp(2)}$&$n \geqslant 0$&$\Sp(2)$\\
    (S$7'$)&$\Sp(2n+2)\,\Sp(4)$&$\Sp(2n)\,\Sp(2)\,\ul{\Sp(2)}$&$n \geqslant 0$&$\Gm$\\
    (S$8'$)&$\Sp(2m+2)\,\Sp(2n+2)$&$\Sp(2m)\,\ul{\Sp(2)}\,\Sp(2n)$&$m,n \geqslant 0$&$\Gm$\\	
    (S$9'$)&$\Sp(8)$&$\G_2\,\ul{\Sp(2)}$&$p=2$&$\Gm$
  \end{tabular}	
\end{center}

Here, the factor $H_0$ is underlined. Note that we can ignore
(S$0'$) since we may assume that both factors $(G_i,H_i)$ are
non-trivial.

Assume first that $H_0\sim\Sp(2)$. Then, by \Cref{lemma:hard}, the pair $(G,H')$ is
spherical if and only if $\Sp(2)$ is a spherical
$(L_1\times L_2)$-variety, where $L_i$ is the image of $L(G_i,H_i)$ in
the $H_0$-factor of $H_i$. It is easily verified that $L_i\cong \Gm$
in cases (S$4'$)--(S$9'$) except for case
(S$6'$), where $L_i\cong\Sp(2)$. This implies that the fusion
$(G,H')$ of two cases among (S$4'$)--(S$9'$) is spherical
if and only if one of the factors is of type (S$6'$). To show
that in this case $(G,H')$ is indeed known spherical, i.e., that
it appears in our tables, we need to show that $(G,H')$ is independent
of the gluing isogenies $\rho_1$ and $\rho_2$. Note that all
self-isogenies of $\Sp(2)$ are powers of the Frobenius morphism
$F_{H_0}$. Thus, $\rho_1=F_{H_0}^{a_1}$ and
$\rho_2=F_{H_0}^{a_2}$. Then, the isogeny
$\tau=(F_{G_1}^{a_2},F_{G_2}^{a_1})$ of $G_1\times G_2$ yields an
isogeny from $(G,H')$ to a known spherical pair.

Assume now that $H_0\sim\Sp(4)\sim\SOO(5)$, which occurs for case (S$3'$) and
case (S$2'$) with $n=5$. In all these cases, the projection of $L(G_i,H_i)$
is isogenous to $\Sp(2)\times\Sp(2)$. Since the action of
$(\Sp(2)\times\Sp(2))^2$ on $\Sp(4)$ is not spherical (see \eqref{eq:2}),
the pair $(G,H')$ is also not spherical.

Next, let $H_0\sim\SL(4)\sim\SOO(6)$. This occurs for (S$1'$) with
$n=4$ and (S$2'$) with $n=5$. Since $L_1\sim T^1\SL(3)$ or
$\SOO(5)$, the action of $L_1\times L_2$ on $H_0$ is not spherical due
to the dimension bound.

The same holds for all other cases of type (S$1'$) or (S$2'$), which concludes the proof of \ref{it:A2} and therefore \Cref{lem:1}.

\bibliographystyle{amsalpha}


\setcounter{myrow}{-1}

\begin{table}
\caption{Indecomposable spherical pairs $H \subset G$ for $G$ semisimple.}
	\label{Tab:BMKR}
        \renewcommand{\arraystretch}{2.5}
$\begin{array}{>{\stepcounter{myrow}\text{(S\themyrow})\quad}l l l l}
	\multicolumn{1}{c}{\hspace{30pt}(G,\ H)\hfill} && L(G/H)\\
	\hline
	\hline
	\rule{0pt}{4ex}%
	\begin{tikzpicture}[baseline=(current bounding box.center), scale=0.6, every node/.style={inner sep=2pt}]
		\draw (0,1) -- (1,0);
		\draw (1,0) -- (2,1);
		\fill (0,1) circle (0.1) node[above] {$\scriptstyle H$};
		\fill (2,1) circle (0.1) node[above] {$\scriptstyle H$};
		\fill (1,0) circle (0.1) node[below] {$\scriptstyle H$};
	\end{tikzpicture}& \scriptstyle H\text{ \scriptsize simple}
	& \text{max.~torus}
	\\[4ex]
	
	\begin{tikzpicture}[baseline=(current bounding box.center), scale=0.6, every node/.style={inner sep=2pt}]
		\draw (0,0) -- (0,1);
		\draw (0,1) -- (1,0);
		\draw (1,0) -- (2,1);
		\fill (0,0) circle (0.1) node[below] {$\scriptstyle \ft^1$};
		\fill (0,1) circle (0.1) node[above] {$\scriptstyle \SL(n{+}1)$};
		\fill (2,1) circle (0.1) node[above] {$\scriptstyle \SL(n)$};
		\fill (1,0) circle (0.1) node[below] {$\vrule height 8pt width0pt depth0pt\scriptstyle \SL(n)$};
	\end{tikzpicture}& \scriptstyle n \geqslant 2
	& 1
	\\[4ex]
	
	\begin{tikzpicture}[baseline=(current bounding box.center), scale=0.6, every node/.style={inner sep=2pt}]
		\draw (0,1) -- (1,0);
		\draw (1,0) -- (2,1);
		\fill (0,1) circle (0.1) node[above] {$\scriptstyle \SOO(n{+}1)$};
		\fill (2,1) circle (0.1) node[above] {$\scriptstyle \SOO(n)$};
		\fill (1,0) circle (0.1) node[below] {$\scriptstyle \SOO(n)$};
	\end{tikzpicture}& \scriptstyle n \geqslant 5
	& 1
	\\[4ex]
	
	\begin{tikzpicture}[baseline=(current bounding box.center), scale=0.6, every node/.style={inner sep=2pt}]
		\draw (0,0) -- (1,1);
		\draw (1,1) -- (2,0);
		\draw (2,0) -- (3,1);
		\fill (0,0) circle (0.1) node[below] {$\scriptstyle \Sp(2n)$};
		\fill (1,1) circle (0.1) node[above] {$\scriptstyle \Sp(2n{+}4)$};
		\fill (2,0) circle (0.1) node[below] {$\scriptstyle \Sp(4)$};
		\fill (3,1) circle (0.1) node[above] {$\scriptstyle \Sp(4)$};
	\end{tikzpicture}& \scriptstyle n \geqslant 1
	& \Sp(2n-4)
	\\[4ex]
	
	\begin{tikzpicture}[baseline=(current bounding box.center), scale=0.6, every node/.style={inner sep=2pt}]
		\draw (0,0) -- (1,1);
		\draw (1,1) -- (1,0);
		\draw (1,1) -- (2,0);
		\draw (2,0) -- (3,1);
		\draw (3,1) -- (4,0);
		\fill (0,0) circle (0.1) node[below] {$\scriptstyle \SL(m)$};
		\fill (1,1) circle (0.1) node[above, xshift=-0.1cm] {$\scriptstyle \SL(m{+}2)$};
		\fill (1,0) circle (0.1) node[below] {$\scriptstyle \ft^1$};
		\fill (2,0) circle (0.1) node[below] {$\scriptstyle \Sp(2)$};
		\fill (3,1) circle (0.1) node[above, xshift=0.1cm] {$\scriptstyle \Sp(2n{+}2)$};
		\fill (4,0) circle (0.1) node[below] {$\scriptstyle \Sp(2n)$};
	\end{tikzpicture}&\def\arraystretch{1}
	\begin{array}{r} \scriptstyle m \geqslant 1 \\[-0.5ex] \scriptstyle n \geqslant 0 \end{array}
	& \ft^1\cdot\SL(m-2)\times \Sp(2n-2)
	\\[4ex]
	
	\begin{tikzpicture}[baseline=(current bounding box.center), scale=0.6, every node/.style={inner sep=2pt}]
		\draw (0,0) -- (1,1);
		\draw (1,1) -- (2,0);
		\draw (2,0) -- (3,1);
		\draw (3,1) -- (4,0);
		\fill (0,0) circle (0.1) node[below] {$\scriptstyle \SL(m)$};
		\fill (1,1) circle (0.1) node[above, xshift=-0.1cm] {$\scriptstyle \SL(m{+}2)$};
		\fill (2,0) circle (0.1) node[below] {$\scriptstyle \Sp(2)$};
		\fill (3,1) circle (0.1) node[above, xshift=0.1cm] {$\scriptstyle \Sp(2n{+}2)$};
		\fill (4,0) circle (0.1) node[below] {$\scriptstyle \Sp(2n)$};
	\end{tikzpicture}&\def\arraystretch{1}
	\begin{array}{r} \scriptstyle m \geqslant 3 \\[-0.5ex] \scriptstyle n \geqslant 0 \end{array}
	& \SL(m-2)\times \Sp(2n-2)
	\\[4ex]
	
	\begin{tikzpicture}[baseline=(current bounding box.center), scale=0.6, every node/.style={inner sep=2pt}]
		\draw (0,0) -- (1,1);
		\draw (1,1) -- (2,0);
		\draw (2,0) -- (3,1);
		\draw (3,1) -- (4,0);
		\fill (0,0) circle (0.1) node[below] {$\scriptstyle \Sp(2m)$};
		\fill (1,1) circle (0.1) node[above, xshift=-0.2cm] {$\scriptstyle \Sp(2m{+}2)$};
		\fill (2,0) circle (0.1) node[below] {$\scriptstyle \Sp(2)$};
		\fill (3,1) circle (0.1) node[above, xshift=0.2cm] {$\scriptstyle \Sp(2n{+}2)$};
		\fill (4,0) circle (0.1) node[below] {$\scriptstyle \Sp(2n)$};
	\end{tikzpicture}&\def\arraystretch{1}
\begin{array}{r} \scriptstyle m \geqslant 0 \\[-0.5ex] \scriptstyle n \geqslant 0 \end{array}
	&\Sp(2m-2)\cdot\ft^1\cdot\Sp(2n-2)
	\\[4ex]
	
	\begin{tikzpicture}[baseline=(current bounding box.center), scale=0.6, every node/.style={inner sep=2pt}]
		\draw (0,0) -- (1,1) -- (2,0) -- (3,1) -- (4,0) -- (5,1) -- (6,0);
		\fill (0,0) circle (0.1) node[below] {$\scriptstyle \Sp(2m)$};
		\fill (1,1) circle (0.1) node[above] {$\scriptstyle \Sp(2m{+}2)$};
		\fill (2,0) circle (0.1) node[below] {$\scriptstyle \Sp(2)$};
		\fill (3,1) circle (0.1) node[above] {$\scriptstyle \Sp(4)$};
		\fill (4,0) circle (0.1) node[below] {$\scriptstyle \Sp(2)$};
		\fill (5,1) circle (0.1) node[above] {$\scriptstyle \Sp(2n{+}2)$};
		\fill (6,0) circle (0.1) node[below] {$\scriptstyle \Sp(2n)$};
	\end{tikzpicture}&\def\arraystretch{1}
	\begin{array}{r} \scriptstyle m \geqslant 0 \\[-0.5ex] \scriptstyle n \geqslant 0 \end{array}
	&\Sp(2m-2)\times\Sp(2n-2)
	\\[4ex]
	
	\begin{tikzpicture}[baseline=(current bounding box.center), scale=0.6, every node/.style={inner sep=2pt}]
		\draw (0,0) -- (1,1);
		\draw (1,1) -- ++(4,-1);
		\draw (2,0) -- ++(1,1);
		\draw (3,1) -- ++(2,-1);
		\draw (4,0) -- ++(1,1);
		\draw (5,1) -- ++(0,-1);
		\fill (0,0) circle (0.1) node[below] {$\scriptstyle \Sp(2l)$};
		\fill (1,1) circle (0.1) node[above, xshift=-0.4cm] {$\scriptstyle \Sp(2l{+}2)$};
		\fill (2,0) circle (0.1) node[below] {$\scriptstyle \Sp(2m)$};
		\fill (3,1) circle (0.1) node[above] {$\scriptstyle \Sp(2m{+}2)$};
		\fill (4,0) circle (0.1) node[below, xshift=-0.2cm] {$\scriptstyle \Sp(2n)$};
		\fill (5,1) circle (0.1) node[above, xshift=0.4cm] {$\scriptstyle \Sp(2n{+}2)$};
		\fill (5,0) circle (0.1) node[below, xshift=0.2cm] {$\scriptstyle \Sp(2)$};
	\end{tikzpicture}&\def\arraystretch{1}
	\begin{array}{r} \scriptstyle l \geqslant 0 \\[-0.5ex] \scriptstyle m \geqslant 0 \\[-0.5ex] \scriptstyle n \geqslant 0 \end{array}
	&\Sp(2l-2)\times\Sp(2m-2)\times\Sp(2n-2)
	\\[4ex]
	
	\hline
	\begin{tikzpicture}[baseline=(current bounding box.center), scale=0.6, every node/.style={inner sep=2pt}]
		\draw (0,0) -- (1,1);
		\draw (1,1) -- (2,0);
		\draw (2,0) -- (3,1);
		\draw (3,1) -- (4,0);
		\fill (0,0) circle (0.1) node[below] {$\scriptstyle \G_2$};
		\fill (1,1) circle (0.1) node[above, xshift=-0.2cm] {$\scriptstyle \Sp(8)$};
		\fill (2,0) circle (0.1) node[below] {$\scriptstyle \Sp(2)$};
		\fill (3,1) circle (0.1) node[above, xshift=0.2cm] {$\scriptstyle \Sp(2n{+}2)$};
		\fill (4,0) circle (0.1) node[below] {$\scriptstyle \Sp(2n)$};
	\end{tikzpicture}&\def\arraystretch{1}
	\begin{array}{l}\color{red}\scriptstyle\operatorname{char}k=2\\[-0.5ex] \scriptstyle n \geqslant 0   \end{array}
	& \Sp(2n-2)
	\\[3ex]
	\hline
\end{array}
$
\end{table}

\end{document}